\documentclass[preprint,12pt]{elsarticle}

\usepackage{hyperref}
\usepackage{paralist, mathdots} 
\usepackage{amssymb} 
\usepackage{amsmath,tikz,wasysym,multirow} 
\usepackage{xfrac}
\usepackage{graphicx}
\usepackage{makeidx} 
\usepackage[T1]{fontenc}

\usepackage{caption}
\usepackage{subcaption}

\newtheorem{theorem}{Theorem}[section]
\newtheorem{example}{Example}[section]

\newtheorem{remark}{Remark}[section]
\newtheorem{definition}{Definition}[section]
\newtheorem{corollary}{Corollary}[section]
\newtheorem{lemma}{Lemma}[section]
\newenvironment{proof}{\textbf{Proof.}}{\qquad $\Box$ \bigskip }

\newcommand{\R}{\mathbb{R}}
\newcommand{\C}{\mathbb{C}}
\newcommand{\N}{\mathbb{N}}
\newcommand{\mc}{\mathcal}
\newcommand{\ii}{\mathrm{i}}
\newcommand{\F}{\mathcal{F}}
\newcommand{\K}{\mathcal{K}}
\newcommand{\im}{\mathrm{i}}
\newcommand{\hK}{\hat{\mc K}}
    \definecolor{helena}{rgb}{.2,.8,.4}
    \definecolor{steve}{rgb}{.8,.2,.2}
    \definecolor{todo}{rgb}{.2,.2,.8}

\newcommand{\st}[1]{%
{\textcolor{steve}{\sf\underline{S:}\ #1}
}}
\newcommand{\he}[1]{%
\textcolor{helena}{\sf\underline{H:}\ #1}
}

\journal{Journal of Mathematical Analysis and Applications}

\begin{document}
\begin{frontmatter}
\title{The Karpelevi\v c Region Revisited}

\author[MT]{Stephen Kirkland\corref{FundingThanks}}
\ead{stephen.kirkland@umanitoba.ca}
\author[UCD]{Thomas Laffey}
\ead{thomas.laffey@ucd.ie}
\author[UCD]{Helena \v Smigoc}
\ead{helena.smigoc@ucd.ie}

\cortext[FundingThanks]{The authors' work was supported by University College Dublin under Grant SF1588. S.K.'s research is supported 
	in part by NSERC Discovery Grant RGPIN--2019--05408.}

\address[MT]{Department of Mathematics, University of Manitoba, Winnipeg, MB, Canada}
\address[UCD]{School of Mathematics and Statistics, University College Dublin, Belfield, Dublin 4, Ireland}

\begin{abstract} We consider the  Karpelevi\v c region $\Theta_n \subset \C$ consisting of all eigenvalues of all stochastic matrices of order $n$. We provide an alternative  characterisation of $\Theta_n$ that sharpens the original description given by Karpelevi\v c. In particular, for each $\theta \in [0, 2\pi),$ we identify the point on the boundary of $\Theta_n$ with argument $\theta.$ We further prove that if $n \in \N$ with $n \ge 2,$ and  $t \in \Theta_n,$ then $t$ is a subdominant eigenvalue of some stochastic matrix of order $n$.   
\end{abstract}

\end{frontmatter}

\section{Introduction}

An entrywise nonnegative matrix $T$ is {\emph{stochastic}} if each of its row sums is equal to $1$. To any stochastic matrix $T$ we may associate a corresponding Markov chain, and  there is a wealth of literature on the eigenvalues of stochastic matrices, in part because those eigenvalues govern the convergence (or lack thereof) of the associated Markov chain. From the Perron--Frobenius theorem, it follows immediately that any eigenvalue of a stochastic matrix sits inside the unit disc, and the problem (posed originally by Kolmogorov \cite{Kolmogorov}) of determining the set 
$$\Theta_n = \{t \in \mathbb{C}| t \mbox{{\rm{ is an eigenvalue of a stochastic matrix of order }}} n \}$$  arises naturally. It is not difficult to show that $\Theta_n$ is star--shaped with respect to the origin, and consequently it suffices to describe the boundary of $\Theta_n,$ which we denote by $\partial\Theta_n,$ in order to completely characterise that  region.

A classic result of Karpelevi\v c \cite{Ka} describes $\Theta_n$, and a simplification of the original statement is given in \cite{MR1077984} and \cite{I}.  We present a result from \cite{I} below, after introducing a couple of relevant definitions. 

\begin{definition}
Given $n \in \mathbb{N}$, the set $$\F_n=\{\sfrac{p}{q}; 0 \leq p < q \leq n, \gcd(p,q)=1\}$$ is called \emph{the set of Farey fractions of order $n$}. 
\end{definition}

\begin{definition}
The pair $(\sfrac{p}{q},\sfrac{r}{s})$ is called \emph{a Farey pair (of order $n$)}, if $\sfrac{p}{q},\sfrac{r}{s} \in \mc F_n$, $\sfrac{p}{q} < \sfrac{r}{s}$, 
and $\sfrac{p}{q}< x <\sfrac{r}{s}$ implies $x \not \in \mc F_n$.  

The Farey fractions $\sfrac{p}{q}$ and $\sfrac{r}{s}$  are called \emph{Farey neighbours}, if one of  $(\sfrac{p}{q},\sfrac{r}{s})$ and $(\sfrac{r}{s}, \sfrac{p}{q})$ is a Farey pair.  
\end{definition}

\begin{theorem}(Karpelevi\v c, \cite{Ka}, \cite{Kaenglish}, \cite{I})\label{thm:Karpelevic}
The region $\Theta_n$ is symmetric with respect to the real axis, is included in the unit disc $\{z \in \C, |z|\leq 1\},$ and intersects the unit circle  $\{z \in \C, |z|=1\}$ at the points $\{e^{ \frac{2 \pi \ii p}{q}}, \sfrac{p}{q} \in \F_n\}$.  The boundary of $\Theta_n$ consists of these points and of curvilinear arcs connecting them in circular order. 

Let the endpoints of an arc be $e^{2 \pi \ii p \over q}$ and $e^{2 \pi \ii r \over s}$ with $q<s$. Each of these arcs is given by the following parametric equation: 
 \begin{equation}\label{eq:Karpelevic Poly}
 t^s (t^q-\beta)^{\lfloor {n \over q} \rfloor}=\alpha^{\lfloor {n \over q} \rfloor}t^{q\lfloor {n \over q}\rfloor}, \, \alpha \in [0,1], \, \beta \equiv 1-\alpha.
 \end{equation}
\end{theorem}

At first glance, Theorem \ref{thm:Karpelevic} appears to give a pretty complete description of $\Theta_n$: for Farey neighbours $\sfrac{p}{q}, \sfrac{r}{s}$ with $q<s,$ and an angle $\theta$ between $\sfrac{2 \pi p}{q}$ and $\sfrac{2 \pi r}{s},$ we simply  find an appropriate $\alpha \in [0,1],$ and a root ${t}$ of \eqref{eq:Karpelevic Poly},  so that the ray of complex numbers with argument  $\theta$ intersects $\partial \Theta_n$ at ${t}$. It is natural to wonder how to choose $\alpha$ in terms of $\theta,$ and in Section \ref{rep_sec} we will clarify that point. However, as the following example illustrates, there may be more than one candidate root of \eqref{eq:Karpelevic Poly} to consider.

\begin{example}\label{Theta12} 
	Suppose that $n=12$, and note that $(\frac{3}{10}, \frac{1}{3})$ is a Farey pair of order $12$. According to Theorem \ref{thm:Karpelevic}, each point on the boundary of $\Theta_{12}$ in the sector of complex numbers with arguments in $[\frac{3 \pi}{5}, \frac{2 \pi}{3}]$ is a root of the polynomial $f_\alpha(t)=(t^3 -\beta)^{4} - \alpha^4 t^2$ for some $\alpha \in [0,1]$.  In this example, we illustrate how Theorem \ref{thm:Karpelevic} leaves open the possibility that there may be more than one root of $f_\alpha$ in that sector. 
			
			First note $$f_\alpha(\lambda^2)=(\lambda^6 - \alpha \lambda -\beta)(\lambda^6 +\alpha \lambda -\beta)(\lambda^6 - i\alpha \lambda -\beta)(\lambda^6 + i\alpha \lambda -\beta),$$
			and
			 for each $\alpha \in [0,1]$, there is a root $\lambda_0$ of the polynomial $\lambda^6 - i\alpha \lambda -\beta$ 
			whose argument lies in  $[\frac{13 \pi}{10}, \frac{4\pi}{3}]$, such that $\lambda_0^2$ is a root of $f_\alpha$ that lies  in the sector of complex numbers with arguments in $[\frac{3 \pi}{5}, \frac{2 \pi}{3}].$ (As our results in Section \ref{rep_sec} will show, it turns out that $\lambda_0^2$  is on $\partial \Theta_{12}$.)
			
			However, it is also the case that for each $\theta \in [\frac{6 \pi}{5}, \frac{4 \pi}{3}],$ there is a root $\lambda_1$ of the polynomial  $\lambda^6 - \alpha \lambda -\beta$ such that $\lambda_1^2 $ has argument $2\theta$ and is also a root of $f_\alpha$. So while $f_\alpha$ always furnishes a root that is a boundary point of $\Theta_{12}$ in the sector of interest, there is a range of values of $\alpha$ (roughly the interval $[0, 0.3986]$) for which $f_\alpha$ has two roots in the sector of interest.  
			
			Figure  \ref{fig:test} (\subref{fig:sub1}) illustrates the situation. The black curve depicts the locus of roots associated with $\lambda^6 - i\alpha \lambda -\beta, \alpha \in [0,1],$ while the blue curve shows the locus of roots associated with $\lambda^6 - \alpha \lambda -\beta, \alpha \in [0,1].$ The green, red, and blue circles denote $e^{\frac{2 \pi \ii}{3}}, e^{\frac{3 \pi \ii}{5}} $ and $e^{\frac{2 \pi \ii}{5}},$ respectively. The blue curve exits the sector of complex numbers with arguments in $[\frac{3 \pi}{5}, \frac{2 \pi}{3}] $   when the corresponding value of  $\alpha$ is  approximately $0.3986.$  
\end{example}



Example \ref{Theta12} reveals a further subtlety inherent in the statement of Theorem \ref{thm:Karpelevic}: given Farey neighbours $\sfrac{p}{q},\sfrac{r}{s},$  there may be more than one root of \eqref{eq:Karpelevic Poly} in the corresponding sector in the complex plane, and Theorem \ref{thm:Karpelevic} does not identify which of those roots lies on $\partial \Theta_n$. In Section \ref{rep_sec}, we  resolve that ambiguity by proving the following theorem.

%

\begin{figure}[h]\label{12example_fig}
\centering
\begin{subfigure}{.5\textwidth}
  \centering
  \includegraphics[width=0.9\linewidth]{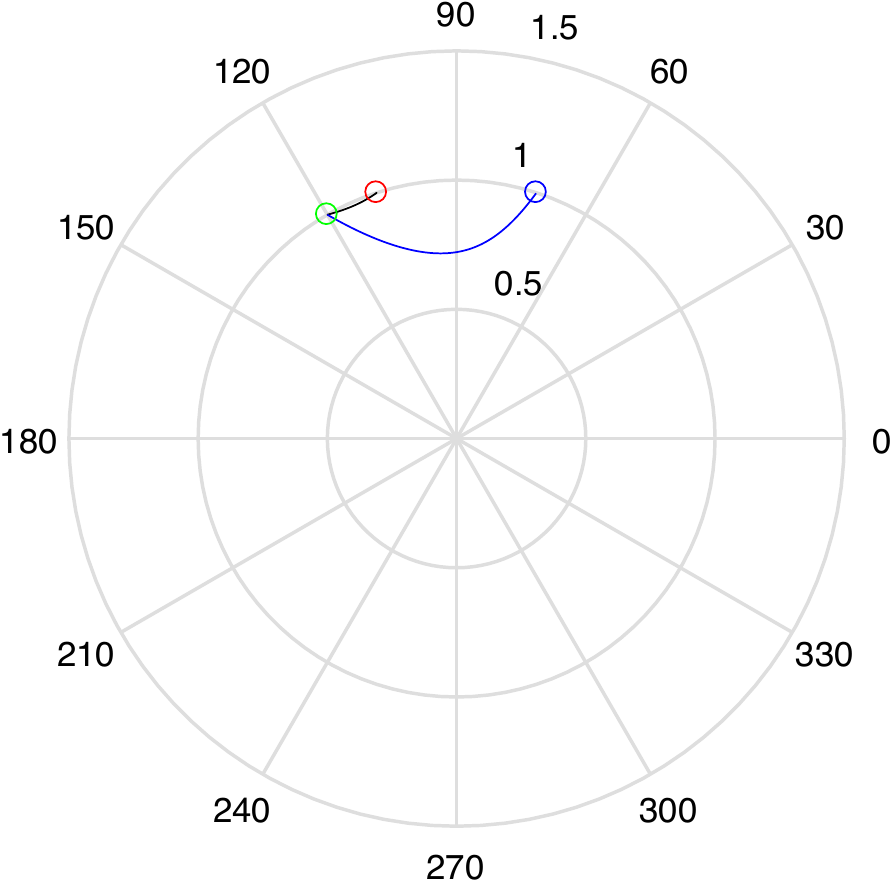}
\caption{Loci of two roots of the polynomial $(t^3 -\alpha)^{4} - \alpha^4 t^2$ with arguments in   $[\frac{3 \pi}{5}, \frac{2 \pi}{3}].$  }
  \label{fig:sub1}
\end{subfigure}%
\begin{subfigure}{.5\textwidth}
  \centering
  \includegraphics[width=0.9\linewidth]{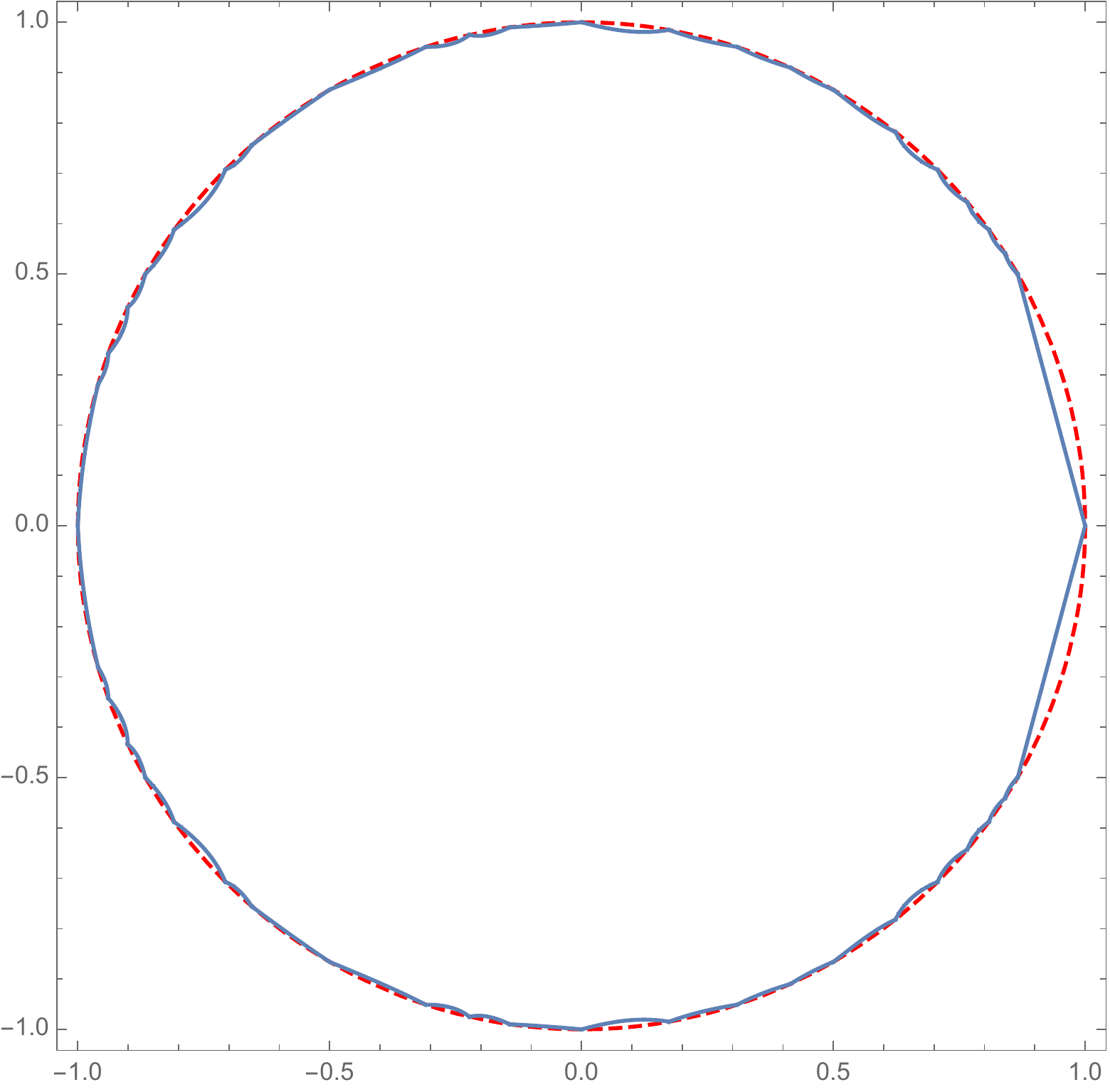}
 \caption{Full region}
  \label{fig:sub2}
\end{subfigure}
\caption{\label{fig:test} $\Theta_{12}$}
\end{figure}


\begin{theorem}\label{thm:gcd200}
Suppose that  $(\frac{p}{q},\frac{r}{s})$ is a Farey pair in $\F_n$, $q\leq s$, $d\equiv \left \lfloor \frac{n}{q}\right \rfloor$ and $\delta\equiv \gcd(d,s)$. Define 
$s_1, d_1, r_1$ and $j_0 \in \{0,\ldots, \delta-1\}$ via the equations $s=s_1 \delta$, $d=d_1\delta$, $r=r_1\delta+j_0$.
Furthermore, we define $\hat r \in \{0,\ldots, s_1-1\}$ and $l_0 \in \{0,\ldots, d_1-1\}$ to be the nonnegative integers that solve the equation $r_1 =d_1  \hat r -l_0  s_1 $. 

For $\theta \in [\sfrac{2 \pi p}{q},\sfrac{2 \pi r}{s}]$ the point on the boundary of $\Theta_n$ with argument $\theta$ is given by $\hat \rho^{d_1} e^{\ii \theta}$ where $\hat \rho$ is the unique positive solution to 
$$\hat \rho^{s_1} \sin(q d_1 \hat \theta)- \hat \rho^{qd_1} \sin \left(s_1 \hat \theta -\frac{2 \pi j_0}{\delta d_1} \right)-\sin\left((qd_1 - s_1)\hat\theta +\frac{2 \pi j_0}{\delta d_1} \right)=0,$$
for $\hat \theta = \frac{1}{d_1}(\theta+2 \pi l_0)$. Furthermore,  $\hat \rho^{d_1} e^{\ii \theta}$ is a root of the rational function $\phi_{\alpha}(t) = (t^q-\beta)^d-\alpha^dt^{qd-s}$,  where $\alpha$ is given by $$\alpha \sin\left((qd_1 - s_1)\hat\theta +\frac{2 \pi j_0}{\delta d_1} \right)=\hat \rho^{s_1} \sin(q d_1 \hat \theta).$$ 
\end{theorem} 

The following example 
 illustrates an application of Theorem \ref{thm:gcd200}. 

\begin{example}\label{7pi/12}{\rm{Consider the complex number $z=0.9e^{\frac{7 \pi \ii}{12}}.$  Evidently $z \notin \Theta_2$ since the latter  is a subset of $\R$,   and certainly $z \in \Theta_{24},$ since $\Theta_{24}$ is star--shaped with respect to the origin, and contains $e^{\frac{7 \pi \ii}{12}}$. What is the smallest value of $n$ for which $z \in \Theta_n$? In this example we show how Theorem \ref{thm:gcd200} can be used to answer that question.  

First we test whether or not $z \in \Theta_5$. Since $\frac{7 \pi }{12} \in [\frac{\pi}{2}, \frac{2 \pi}{3}], $ we need to consider the Farey neighbours $\frac{1}{4}, \frac{1}{3}$ in $\F_5.$ According to Theorem \ref{thm:gcd200}, when $\theta \in [\frac{\pi}{2}, \frac{2 \pi}{3}],$ the corresponding boundary point of $\Theta_5$ is given by $\rho e^{\ii \theta},$ where $\rho$ is the unique positive solution to $\rho^4 \sin(3 \theta)-\rho^3 \sin(4 \theta) + \sin(\theta)=0.$ For the specific choice of $\theta = \frac{7 \pi }{12},$ that equation becomes $\frac{-1}{\sqrt{2}} \rho^4 - \frac{\sqrt{3}}{2}\rho^3 +\frac{\sqrt{6}+\sqrt{2}}{4} = 0.$ Since $\frac{-1}{\sqrt{2}} (0.9)^4 - \frac{\sqrt{3}}{2} (0.9)^3 +\frac{\sqrt{6}+\sqrt{2}}{4} < 0,$ we find that the boundary point of $\Theta_5$ with argument $\frac{7 \pi }{12}$ has modulus less than $0.9,$ and so we conclude that $z \notin \Theta_5.$ 

Next we consider $\Theta_6.$ Again we have the Farey neighbours $\frac{1}{4}, \frac{1}{3}$ in $\F_6,$ and 
when $\theta \in [\frac{\pi}{2}, \frac{2 \pi}{3}],$ the boundary point of $\Theta_6$ with argument $\theta$ is  $\rho e^{\ii \theta},$ where $\rho$ is the positive solution to $\rho^3 \sin(2 \theta)+\rho^2 \sin(3 \theta) + \sin(\theta)=0.$ For $\theta = \frac{7 \pi }{12},$ that equation is equal to $\frac{-1}{2} \rho^3 -\frac{1}{\sqrt{2}}\rho^2 + \frac{\sqrt{6}+\sqrt{2}}{4} =0.$ Observing that $\frac{-1}{2} (0.9)^3 -\frac{1}{\sqrt{2}}(0.9)^2 + \frac{\sqrt{6}+\sqrt{2}}{4} >0,$ we see that the boundary point of $\Theta_6$ with argument $\frac{7 \pi }{12}$ has modulus greater than $0.9$. Hence $z \in \Theta_6.$ Thus we find that $z \in \Theta_n$ if and only if $n\ge 6.$ 
}}
\end{example} 

Recall that for a stochastic matrix $T$, a subdominant eigenvalue is one whose modulus is next largest after that of the Perron eigenvalue $1$; since the eigenvalues of $T$ are, in general, complex numbers, it is possible that a given stochastic matrix may have several subdominant eigenvalues. It is also possible (if $T$ is irreducible and periodic, for example) that a subdominant eigenvalue may have modulus $1$. There is a natural interest in the subdominant eigenvalues of a stochastic matrix $T$, as they  govern the asymptotic rate of convergence (or lack of it) of the associated Markov chain. In Section \ref{sec:sub}, we prove a companion result to  Theorem  \ref{thm:gcd200}, which shows that for each $t \in \Theta_n,$ there is a stochastic matrix of order $n$ having $t$ as a subdominant eigenvalue.

\newpage

\section{Farey Pairs}

In this section we fix the notation associated with Farey pairs, and state a selection of easy or well--known results on Farey pairs. Our first result offers a characterisation of Farey pairs; the proof of the result can be found for example in \cite{MR2445243}. 

\begin{lemma}
Let $\sfrac{p}{q},\sfrac{r}{s} \in \F_n$. Then $(\sfrac{p}{q},\sfrac{r}{s})$ is a Farey pair of order $n$ if and only if $q+s>n$ and $qr-ps=1$.
\end{lemma}

Since the region $\Theta_n$ is symmetric with respect to the real axis, it is enough to study only one of the two regions (Farey pairs) that are conjugate to each other. The next lemma explains the relationship between the two Farey pairs that correspond to conjugate regions.

\begin{lemma}\label{lem:conjugate pairs}
$(\sfrac{p}{q},\sfrac{r}{s})$ is a Farey pair of order $n$  if and only if $(\sfrac{(s-r)}{s},\sfrac{(q-p)}{q})$ is a Farey pair of order $n$, and those are the only two pairs of order $n$ with denominators $q$ and $s$. Moreover, the sectors 
$\{ z \in \C| \arg(z) \in (\sfrac{2 \pi p}{q},\sfrac{2 \pi r}{s}) \}$ and $\{ z \in \C| \arg(z) \in (\sfrac{2 \pi (s-r)}{s},\sfrac{2 \pi (q-p)}{q}) \} $ 
 corresponding to the two pairs are complex conjugates of each other. 
\end{lemma}

In particular, Lemma \ref{lem:conjugate pairs} allows us to  work with the pair $(\sfrac{p}{q},\sfrac{r}{s})$ that satisfies $q<s$, and we do so throughout the sequel.   Next we introduce a few definitions associated with Farey pairs upon which we will depend throughout the paper. 
%

\begin{definition}
Let $n$, $q$, and $s$ be positive integers such that $(\sfrac{p}{q},\sfrac{r}{s})$ is a Farey pair in $\F_n$ for some $p$ and $r$. Then:
\begin{enumerate}
\item $\mc K_n(\sfrac{p}{q},\sfrac{r}{s})\equiv \partial\Theta_n \cap  \{z \in \C; \arg(z) \in (\sfrac{2 \pi p}{q},\sfrac{2 \pi r}{s})\}$,
\item $\hat{\mc K}_n(q,s)\equiv \K_n(\sfrac{p}{q},\sfrac{r}{s}) \cup \K_n(\sfrac{(q-r)}{s},\sfrac{(q-p)}{q})$,
\item $\arg(q,s)\equiv (\sfrac{2 \pi p}{q},\sfrac{2 \pi r}{s}) \cup (\sfrac{2 \pi  (q-r)}{s}, \sfrac{2 \pi  (q-p)}{q})$.
\end{enumerate}
\end{definition}

\begin{definition}\label{def:deltaetc}
Let $\zeta=(\frac{p}{q}, \frac{r}{s} )$ be a Farey pair in $\F_n$. We set $d(\zeta)\equiv \left \lfloor \frac{n}{q}\right \rfloor$ and $\delta(\zeta)\equiv gcd(d(\zeta),s)$. We define $s_1(\zeta), d_1(\zeta), r_1(\zeta)$ and $j_0(\zeta) \in \{0,\ldots, \delta(\zeta)-1\}$ via the equations 
\begin{align*}
s&=s_1(\zeta) \delta(\zeta), \\ d(\zeta)&=d_1(\zeta) \delta(\zeta), \\ r&=r_1(\zeta)\delta(\zeta)+j_0(\zeta).
\end{align*}
Furthermore, we define $\hat r(\zeta)  \in \{0,\ldots, s_1(\zeta)-1\}$ and $l_0(\zeta) \in \{0,\ldots, d_1(\zeta)-1\}$ to be the nonnegative integers that solve the equation $r_1(\zeta) =d_1(\zeta)  \hat r(\zeta) -l_0(\zeta)  s_1(\zeta) $. 
\end{definition}
Evidently in Definition \ref{def:deltaetc}, it must be the case that 
$r_1=\left\lfloor \frac{r}{\delta} \right \rfloor$ and  $j_0 = r-\delta \left\lfloor \frac{r}{\delta} \right \rfloor$.

As the Farey pair will be invariably known from the context,  throughout the paper  we will suppress the explicit dependence of 
parameters   $d$, $\delta$ $d_1$, $s_1$, $r_1$, $j_0$, $\hat r$ and $l_0$ on $\zeta$.

\begin{lemma}\label{lem:sin}
Let $(\sfrac{p}{q},\sfrac{r}{s})$ be a Farey pair of order $n$ with $2 \leq q<s$, and parameters  $d$, $\delta$ $d_1$, $s_1$, $r_1$, $j_0$, $\hat r$ and $l_0$ as in Definition \ref{def:deltaetc}.  Let $\tau \in [2 \pi (\sfrac{p}{q}), 2 \pi (\sfrac{r}{s})]$ and $\hat \tau \equiv \frac{1}{d_1}(\tau+2 \pi l_0)$. Then the following inequalities hold: 
\begin{enumerate}
\item $\sin(q \tau)=\sin(q d_1 \hat \tau) \geq 0$, with equality if and only if $\tau= \sfrac{2 \pi p}{q}$.
\item $\sin(\frac{1}{d}(s \tau-2 \pi r))=\sin (s_1 \hat \tau -\frac{2 \pi j_0}{\delta d_1} )\leq 0$ with equality only for $\tau= \sfrac{2 \pi r}{s}$. 
\item $\sin(\frac{1}{d}(2 \pi r-s \tau)+q \tau)=\sin((qd_1 - s_1)\hat\tau +\frac{2 \pi j_0}{\delta d_1})>0.$ 
%
%
\end{enumerate}
\end{lemma}
\begin{proof}
The following equalities are readily established from the definition of $\hat \tau$ and Definition \ref{def:deltaetc}:
\begin{align*}
q d_1 \hat \tau&=q \tau +2 \pi l_0 q \\
s_1 \hat \tau-\frac{2 \pi j_0}{\delta d_1}&=\frac{1}{d}(\tau s- 2 \pi r+2 \pi d \hat r) \\
(qd_1 - s_1)\hat\tau +\frac{2 \pi j_0}{\delta d_1}&=\frac{1}{d}(2 \pi r-s \tau)+q \tau+ 2 \pi (l_0 q-\hat r). 
\end{align*}
With this we have proved that all the equalities in 1.--3. hold, and it now suffices to prove the inequalities involving $\tau$. 
To accomplish this, we note:
\begin{equation}\label{eq:qr}
 \frac{2 \pi qr}{s}=2 \pi \left( \frac{ps+1}{s}\right) =2 \pi p+ \frac{2 \pi}{s}< 2 \pi p+\pi 
\end{equation}
and 
\begin{equation}\label{eq:sp}
\frac{2 \pi s p}{q}=2 \pi \left(\frac{qr-1}{q}\right) =2 \pi r- \frac{2 \pi}{q}\geq 2 \pi r-\pi,
\end{equation}
with equality in (\ref{eq:sp}) if and only if $q=2$. 
\begin{enumerate}
\item From \eqref{eq:qr} we get:
$$q\tau \in  [2 \pi p,  \sfrac{2 \pi qr}{s}] =[2 \pi p, 2 \pi p + \sfrac{2 \pi }{s}] \subseteq [2\pi p, 2 \pi p+ \pi),$$
and the inequality in 1. follows. 
\item We have: 
\begin{align*}
\sfrac{(2 \pi r-s \tau)}{d} &\in [\sfrac{2 \pi r}{d}-\sfrac{2 \pi r}{d}, \sfrac{2 \pi r}{d}-\sfrac{2 \pi ps}{qd}]\\
&=[0,\sfrac{(2 \pi ps + 2 \pi - 2 \pi ps)}{qd} ]=[0,\sfrac{ 2 \pi}{qd} ],
\end{align*}
and the inequality in 2. follows. 
\item Suppose for concreteness that $qd>s.$ 
Note that $h(\tau)\equiv \frac{1}{d}(2 \pi r-s \tau)+q \tau = \frac{1}{d}\left(2 \pi r + (qd-s)\tau\right),$ and $qd > s,$ imply that $h(\tau)\in [h(\sfrac{2 \pi p}{q}),h(\sfrac{2 \pi r}{s})]$. Now:
\begin{align*}
h(\sfrac{2 \pi p}{q})&=\frac{1}{d}(2 \pi r+2 \pi pd-(\sfrac{2 \pi ps}{q})) \\
&=\frac{1}{d}(2 \pi r+2 \pi pd-2 \pi r+(\sfrac{2 \pi}{q}))=2 \pi p+(\sfrac{2 \pi}{qd})
\end{align*}
by \eqref{eq:sp}, and 
$$h(\sfrac{2 \pi r}{s})=\frac{1}{d}(2 \pi r-2 \pi r-(\sfrac{2 \pi qdr}{s}))=\sfrac{2 \pi q r}{s}=2 \pi p +\sfrac{2 \pi }{s}$$
 by \eqref{eq:qr}. Hence, 
$$h(\tau)\in [2 \pi p+(\sfrac{2 \pi}{qd}),2 \pi p +\sfrac{2 \pi }{s}]= [\sfrac{2 \pi}{qd},\sfrac{2 \pi }{s}].$$ In the case that $qd<s,$ an analogous argument establishes that $h(\tau)\in  [\sfrac{2 \pi }{s},\sfrac{2 \pi}{qd}]$. 
The desired conclusion now follows. 
\end{enumerate}
\end{proof}


\newpage

\section{Types of Reduced Ito Polynomials}


Let $(\sfrac{p}{q},\sfrac{r}{s})$ be a Farey pair satisfying $q<s$, $d\equiv \lfloor \sfrac{n}{q} \rfloor$, and let us write 
$$n=d q+y,\text{ where }y \in \{0,\ldots, q-1\}.$$
Following the terminology introduced in \cite{JP}, we distinguish between four different types of arcs arising in Theorem \ref{thm:Karpelevic} and the associated \emph{reduced Ito polynomials}.\\
 Type 0: \begin{equation}\label{eq:Type 0} f_{\alpha}(t)=(t-\beta)^n-\alpha^n, \, p=0, \, q=r=1, \, s=n, \end{equation} 
 Type I: 
 \begin{equation}\label{eq:Type I} f_{\alpha}(t)=t^s-\beta t^{s-q}-\alpha, \, d=1,\end{equation}
Type II: \begin{equation}\label{eq:Type II}f_{\alpha}(t)=(t^q-\beta)^d-\alpha^d t^{qd-s}, \, d>1\text{ and }q d > s,\end{equation}
Type III: \begin{equation}\label{eq:Type III}f_{\alpha}(t)=t^{s-qd}(t^q-\beta)^d-\alpha^d,\, d>1\text{ and }q d < s. \end{equation}

Note that a reduced Ito polynomial $f_{\alpha}(t)$ depends only on $\alpha, q$, $s$ and $n$. Type 0 arcs and polynomials are easy to understand, as the roots of a Type 0 polynomial are given by $\alpha e^{\frac{2 \pi \ii j}{n}} + 1-\alpha, j=0,\ldots, n-1.$ Consequently, 
we will focus our investigation on Types I, II, and III. 

%
%

\begin{example}\label{ex:q=2}
Let us look at Farey pairs in $\F_n$ that contain $\sfrac{1}{2}$. We write $n=2m$ for $n$ even, and $n=2m-1$ for $n$ odd. In both cases:
$(\sfrac{1}{2}, \sfrac{m}{(2m-1)})$ is a  Farey pair of order $n$, however, depending on the parity of $n$, we get two different types of reduced Ito polynomials. For $n=2m$ the corresponding Ito polynomial is of Type II:
\begin{equation}\label{eq:q=2 even}
f_{\alpha}(t)=(t^2-\beta)^m-\alpha^m t,
\end{equation}
while for $n=2m-1$ the corresponding Ito polynomial is of Type III:
\begin{equation}\label{eq:q=2 odd}
f_{\alpha}(t)=t(t^2-\beta)^{m-1}-\alpha^{m-1}.
\end{equation}
\end{example}

 To unify discussions regarding different types of polynomials we introduce the following definition. 

\begin{definition}
Let $(\sfrac{p}{q},\sfrac{r}{s})$, $q \geq 2$,  be a Farey pair in $\mc F_n$. Then
\begin{equation}\label{eq:ItoFunction}
\phi_{\alpha}(t)=(t^q-\beta)^d-\alpha^d t^{qd-s}.
\end{equation} 
is \emph{ the Ito rational function corresponding to $\hK_n(q,s)$}. 
\end{definition}

Note that $\phi_{\alpha}(t)=t^{\epsilon}f_{\alpha}(t)$, where $\epsilon=\min\{0,qd-s\}$ and $f_{\alpha}$ is the corresponding reduced Ito polynomial. In particular, the nonzero  roots of the polynomial $f_{\alpha}$ coincide with the roots of $\phi_{\alpha}$.

\section{Boundary of the Karpelevi\v c region}\label{rep_sec}

In this section we aim to better understand the Karpelevi\v c region, and refine the description of the region given in Theorem \ref{thm:Karpelevic}. This task is accomplished by examining the behaviour of roots of $\phi_{\alpha}$.

\begin{lemma}\label{lem:multiple}
Let $\phi_{\alpha}$, $\alpha \in (0,1)$, be the Ito rational function corresponding to a Farey pair $(\sfrac{p}{q},\sfrac{r}{s})$, $q \geq 2$.
If $t_0$ is a multiple root of $\phi_{\alpha}$, then $t_0^s$ is a real number, and $t_0^s \neq 1$. 
\end{lemma}
\begin{proof} 
	Evidently the conclusion holds if $t_0=0,$ so we assume henceforth that $t_0 \ne 0.$ 
If $t_0$ is a multiple root of the Ito rational function \eqref{eq:ItoFunction}, then 
\begin{align*}
&(t_0^q-\beta)^d-\alpha^d t_0^{qd-s}=0, \text{ and }\\
&dq (t_0^q-\beta)^{d-1}t_0^{q-1}=\alpha^d (dq-s) t_0^{qd-s-1}. 
\end{align*}
From the second equation we get 
$$\alpha^d  t_0^{qd-s-1}=\frac{dq}{dq-s} (t_0^q-\beta)^{d-1}t_0^{q-1},$$
and inserting this in the first equation gives us: 
$(t_0^q-\beta)=\gamma t_0^q$, for $\gamma =\frac{dq}{dq-s} \in \R$. Inserting this back in the first equation, we get: 
$\gamma^d t_0^{dq}=\alpha^d t_0^{dq-s},$ or equivalently $t_0^s=(\frac{\alpha}{\gamma})^d \in \R$. Since $\alpha < 1 < \gamma,$ we see that $t_0^s<1.$ 
\end{proof}

\begin{theorem}\label{thm:double}
An Ito rational function $\phi_{\alpha}$, $\alpha \in (0,1)$, corresponding to $\hK_n(q,s)$ does not have any double roots inside the set $\hK_n(q,s)$.  
\end{theorem}

\begin{proof}
The statement is obvious for Type 0 polynomials. Let us assume that $t_0 \in \hK_n(q,s)$ for $q \geq 2$, and that $t_0$ is a multiple root of $\phi_\alpha $. By Lemma \ref{lem:multiple} we have $t_0^s\in \R$, and hence $\arg(t_0)=\frac{\pi k}{s}$ for some positive integer $k$. If $\sfrac{2p}{q}< \sfrac{k}{s} < \frac{2r}{s},$ then $2ps<qk<2qr$, and substituting $ps=qr-1$, we get $2qr-2<qk<2qr$. This implies $qk=2qr-1$, which is clearly impossible. 
%
%
\end{proof}

The corollary below confirms Conjecture 6.2 of \cite{JP}. 

\begin{corollary}\label{cor:diff}
Consider $t(\alpha)\in \partial \Theta_n$ as a root of the corresponding reduced Ito polynomial, $\alpha \in [0,1].$ Then $t(\alpha)$ is differentiable in $\alpha$ for $\alpha \in (0,1)$, and continuous for $\alpha \in [0,1].$
\end{corollary}
\begin{proof}
	The continuity on $[0,1]$ follows from the fact that the coefficients of the corresponding reduced Ito polynomial depend continuous on $\alpha,$ while the differentiability on $(0,1)$ is a consequence of the Implicit Function Theorem. 
\end{proof}

Let $\phi_{\alpha}$, $\alpha \in (0,1)$, be the Ito rational function corresponding to a Farey pair $(\sfrac{p}{q},\sfrac{r}{s})$,  $q \geq 2$. We fix the following notation for the rest of this section: $d\equiv \left \lfloor \frac{n}{q}\right \rfloor$, $\delta\equiv \gcd(d,s)$, $d=d_1\delta$, $s=s_1 \delta$. To better understand the roots of $\phi_{\alpha}$ we consider the following factorisation:
\begin{equation}\label{eq:factorphi}
\phi_{\alpha}(t)=\prod_{j=0}^{\delta-1}g_{\alpha,j}(t),
\end{equation}
where 
\begin{equation}\label{eq:g}
g_{\alpha,j}(t)\equiv (t^q-\beta)^{d_1}-\alpha^{d_1}t^{qd_1-s_1}e^{\frac{2 \pi \ii j}{\delta}}.
\end{equation}
 When $\delta=1$ we have $j=0$ and this reduces to $g_{\alpha,0}(t)\equiv \phi_{\alpha}(t)$. 
It turns out that the roots of $g_{\alpha,j}$ are closely connected with the roots of 
\begin{equation}\label{eq:hg}
\hat g_{\alpha,j}(t)\equiv t^{qd_1}-\beta - \alpha t^{qd_1-s_1}e^{\frac{2 \pi \ii j}{\delta d_1}},
\end{equation}
as it is shown in the following lemma. The functions $g_{\alpha,j}$ and $\hat g_{\alpha,j}$  play a crucial role in the rest of the section, and will always be associated with a particular Farey pair that will be clear from the context. Furthermore, the index $j$ in $g_{\alpha,j}$ and $\hat g_{\alpha,j}$ will be understood modulo $\delta$. 

\begin{lemma}\label{lem:powers}
Let $(\frac{p}{q}, \frac{r}{s})$, $q \geq 2$, be a Farey pair in $\F_n$, and $g_{\alpha,j}$, $\hat g_{\alpha,j}$ as defined in  \eqref{eq:g}  and \eqref{eq:hg}. 
 Let $\hat t_a$, $a=1,\ldots, q d_1$, be the roots of $\hat g_{\alpha,j}(t)$. Then $\hat t_a^{d_1}$, $a=1,\ldots, q d_1$, are the roots of $g_{\alpha,j}(t)$.
\end{lemma}

\begin{proof}
	Denote the roots of $g_{\alpha,j}(t)$ by $t_a, a=1, \ldots, q d_1,$ so that
	$$g_{\alpha,j}(t)=t^{\epsilon}\Pi_{a=1}^{q d_1}(t-t_a)$$
 for $\epsilon=\min\{0,q d_1-s_1\}$, and 
	we have  $g_{\alpha,j}(t^{d_1})=t^{d_1\epsilon}\Pi_{a=1}^{q d_1}(t^{d_1}-t_a)$. On the other hand: 
	\begin{eqnarray*}
		g_{\alpha,j}(t^{d_1}) &=& (t^{q d_1}-\beta)^{d_1 }-\alpha^{d_1 }t^{q d_1^2 -s_1 d_1}e^{\frac{2 \pi i j}{\delta}} \\
		&=& \Pi_{m=0}^{d_1 -1}\left(t^{q d_1}-\beta - \alpha t^{q d_1-s_1}e^{\frac{2 \pi \ii}{ d_1}(\frac{j}{\delta}+m)} \right).
	\end{eqnarray*}
Similarly, 
$$\hat g_{\alpha,j}(t)=t^{\epsilon}\Pi_{b=1}^{q d_1}(t-\hat t_b),$$
and 
$$\hat g_{\alpha,j}(t e^{\frac{2\pi \ii \gamma}{d_1}})=t^{qd_1}-\beta - \alpha t^{qd_1-s_1}e^{\frac{2 \pi \ii }{d_1}(\frac{j}{\delta}-s_1 \gamma)}.$$

	Since $s_1$ and $d_1$ are relatively prime, for each $m=0, \ldots, d_1-1$ there is a $\gamma_m \in \{0, \ldots, d_1-1\}$ such that $-s_1\gamma_m \equiv m \mod d_1$; further, $\{\gamma_0, \ldots, \gamma_{d_1-1}\} =  \{0, \ldots, d_1-1\}.$ 
	Consequently, 
	\begin{eqnarray*}
		g_{\alpha,j}(t^{d_1})& =& \Pi_{m=0}^{d_1-1}\hat g_{\alpha,j}(t e^{\frac{2 \pi \ii \gamma_m}{d_1}}) \\	
		&=&  \Pi_{k=0}^{d_1-1}\hat g_{\alpha,j}(t e^{\frac{2 \pi \ii k}{d_1}}) \\
		&=&\Pi_{k=0}^{d_1-1} (t e^{\frac{2 \pi \ii k}{d_1}})^{\epsilon}\Pi_{b=1}^{q d_1}\left( t e^{\frac{2 \pi \ii k}{d_1}} - \hat t_b\right) \\ 
		&=& t^{d_1 \epsilon}\Pi_{b=1}^{q d_1}\Pi_{k=0}^{d_1-1} \left( t e^{\frac{2 \pi \ii k}{d_1}} - \hat t_b\right)  \\
		&=&t^{d_1 \epsilon} \Pi_{b=1}^{q d_1}\left( t^{d_1} - \hat t_b^{d_1} \right). 
	\end{eqnarray*}
	The conclusion now follows. 
\end{proof}

At this point let us examine the roots  of $g_{\alpha,j}$ and $\hat g_{\alpha,j}$ on $\delta \Theta_n$ for $\alpha=0$ and $\alpha=1$. For $\alpha=0$, $g_{0,j}$ are the same for all $j$, and they all have $e^{\sfrac{2 \pi i p}{q}}$ as a root of multiplicity $d_1$. More interesting is the case, when $\alpha=1$. In the lemma below, the notation in Definition \ref{def:deltaetc} is assumed. 

\begin{lemma}\label{lem:1}
Let $(\frac{p}{q}, \frac{r}{s})$, $2 \leq q < s$, be a Farey pair in $\F_n$, $g_{\alpha,j}$, $\hat g_{\alpha,j}$ as defined in  \eqref{eq:g}  and \eqref{eq:hg}, and $d_1$, $j_0$ and $l_0$ as defined in Definition \ref{def:deltaetc}.   Then $e^{\sfrac{2 \pi i r}{s}}$ is a root of $g_{1,j}$ only for $ j=j_0$, and  
the root $\lambda$ of $\hat g_{1,j_0}$ satisfying $\lambda^{d_1}=e^{\sfrac{2 \pi i r}{s}}$ is equal to $\lambda=e^{\frac{2 \pi i }{d_1}(\frac{r}{s}+l_0)}$.
\end{lemma}

\begin{proof}
As the nonzero roots of $g_{1,j}$ are of the form $e^{\frac{2 \pi i}{s_1} (k \delta+j)}$ for $k=0,\ldots,s_1-1$, it is not hard to see that $g_{1,j_0}$ has $e^{\sfrac{2 \pi i r}{s}}$ as a root, only for $ j_0 \in \{0,\ldots, \delta-1\}$ determined by the equation $$r=r_1 \delta + j_0.$$  

Now we would like to identify the root $\lambda$ of  $\hat g_{1,j_0}$, that satisfies 
\begin{equation}\label{eq:ld}
\lambda^{d_1}= e^{\sfrac{2 \pi i r}{s}}.
\end{equation}
The nonzero roots of $\hat g_{1, j_0}$ are equal to $\lambda_l=e^{\frac{ 2 \pi i}{s_1}(l+\frac{j_0}{d})}$ for $l \in \{0,\ldots, s_1-1\}$. Hence $\lambda_l$ satisfies \eqref{eq:ld} if an only if $l d+j_0$ is congruent to $r$ modulo $s$. This is equivalent to $l d-r_1 \delta$ being an integer multiple of $s$, i.e. $ l d_1-r_1$ being an integer multiple of $s_1$. Let $\hat r  \in \{0,\ldots, s_1-1\}$ and $l_0 \in \{0,\ldots, d_1-1\}$ be the nonnegative integers that solve the equation $r_1=d_1 \hat r-l_0 s_1$, as in Definition \ref{def:deltaetc}. From here we deduce that $$\lambda_{\hat r}=e^{\frac{2 \pi i}{s_1}(\hat r+\frac{j_0}{\delta d_1})}=e^{\frac{2 \pi i }{d_1}(\frac{r}{s}+l_0)}$$ satisfies \eqref{eq:ld}. 
\end{proof}

%

%
%
%


Next lemma illustrates why it is beneficial to consider the roots of $\hat g_{\alpha,j}$ rather than the roots of $g_{\alpha,j}$ directly. 

\begin{lemma}
	\label{lem:phi}
Let $(\frac{p}{q},\frac{r}{s}) \in \F_n$ and $\hat g_{\alpha,j}$ as defined in \eqref{eq:hg} with a nonzero root $\rho e^{\ii  \tau}$. Then $F_{ \tau,j}( \rho)=0$, where
	\begin{equation}\label{eq:qs1II} 
F_{\tau,j}(\rho)\equiv   \rho^{s_1} \sin(q d_1 \tau)-  \rho^{qd_1} \sin \left(s_1  \tau -\frac{2 \pi j}{\delta d_1} \right)-\sin\left((qd_1 - s_1)\tau +\frac{2 \pi j}{\delta d_1} \right),
\end{equation}
and 
\begin{equation}\label{eq:qs2II}
\alpha \sin\left((qd_1 - s_1)\tau +\frac{2 \pi j}{\delta d_1} \right)= \rho^{s_1} \sin(q d_1 \tau).
\end{equation}

Furthermore,  for $j=j_0$, $\theta \in \arg(q,s)$ and 
\begin{equation}\label{eq:theta}
\hat \theta\equiv \frac{1}{d_1}\left(\theta+2 \pi l_0\right)  
\end{equation}
there is a unique positive solution $\hat \rho$ to $F_{\hat \theta,j_0}(\hat \rho)=0$ such that $\hat \rho \in (0,1]$, and in this case $\alpha$ is uniquely defined by \eqref{eq:qs2II}.
\end{lemma}
\begin{proof} 
By considering the real and imaginary parts of the equation $\hat g_{\alpha,j}(\rho e^{\ii \tau})=0,$ then cross--multiplying and simplifying, \eqref{eq:qs1II} is readily established. Evidently \eqref{eq:qs2II} follows directly from the imaginary part of $\hat g_{\alpha,j_0}(\rho e^{\ii \tau})=0.$ 

From Lemma \ref{lem:sin} we find that for $j=j_0$ and $\hat \theta=\frac{1}{d_1}\left(\theta+2 \pi l_0\right)$ the coefficients of $\hat \rho^{s_1}$ and $\hat \rho^{qd_1}$ in \eqref{eq:qs1II} are nonnegative; further,  $F_{\hat \theta,j_0}(0)<0$ and  $F_{\hat \theta,j_0}(1)>0$. Hence, $F_{\hat \theta,j_0}(\hat \rho)$ is strictly increasing as a function of $\hat \rho$, and has a unique positive root in $(0,1]$, as claimed. 
\end{proof}

In the case $d=1$, we have $\delta=1$ and $\phi_{\alpha}=g_{\alpha,0}=\hat g_{\alpha,0}$. In this case Lemma \ref{lem:phi} directly gives us a result for the special case of Type I reduced Ito polynomials. 

\begin{theorem}\label{thm:Type I} 
Suppose that  $(\frac{p}{q},\frac{r}{s}) \in \F_n$ and that $d=1$. 
Let $f_{\alpha}(t)$, $\alpha \in [0,1]$, be the Type I  reduced Ito polynomial   corresponding to $\hK_n(q,s)$. For every $\theta \in \arg(q,s)$ there exist unique $\rho \in (0,1]$ and $\alpha \in [0,1]$ such that $\rho e^{\im\theta}$ is a root of $f_{\alpha}(t)$. Furthermore, $\rho$ satisfies $F_{\theta}(\rho)=0$, where \begin{equation}\label{eq:rho T1}
F_{\theta}(\rho)\equiv \rho^s \sin(q \theta)-\rho^q \sin(s \theta) +\sin((s-q) \theta).
\end{equation} 
In particular, $\rho e^{\ii \theta}$ is the point on $\partial \Theta_n$ with argument $\theta. $ 
\end{theorem}

	Observe that Theorem \ref{thm:Type I}  reparameterises and sharpens  the description of $\partial \Theta_n$ in the sector $\arg(q,s)$ when $\left \lfloor \frac{n}{q} \right \rfloor =1.$ The corresponding Type I polynomial has precisely one root in that sector, and Theorem \ref{thm:Type I}  describes the corresponding boundary arc in terms of the argument $\theta $ rather than in terms of the parameter $\alpha.$ 

\begin{remark}  \label{typeIint}{\rm{ Here we focus on the special case of Type 1 reduced Ito polynomials. The paper \cite{Ki} provides a description of  $L_n,$ the set of all eigenvalues (distinct from $1$) of all $n\times n$ stochastic companion matrices.  (Evidently $L_n \subseteq \Theta_n$.) 
	Lemma 4 of \cite{Ki} considers the possibility that two Type I polynomials, say 
$$	t^{s_1}-\alpha_1 t^{s_1-q_1}-(1-\alpha_1), {\mbox{ and }} \\
	 t^{s_2}-\alpha_2 t^{s_2-q_2}-(1-\alpha_2)$$
have a common root. The conclusion of that result is that if $s_1, s_2 \le n,$ and the two polynomials are distinct -- i.e. either $(s_1,q_1) \ne (s_2,q_2)$, or $(s_1,q_1) = (s_2,q_2)$
and $\alpha_1 \ne \alpha_2$, -- then any common root of those polynomials is either: i) a root of unity; ii) a real number; or iii) in the interior of $L_n$.

Suppose that we have a Farey pair $(\frac{y}{w}, \frac{x}{y} )$ in $\F_n$, and without loss of generality we assume that $w<y$. Suppose further that $t$ is a root of  the Type I polynomial $t^s-\alpha t^{s-q}-(1-\alpha)$, and that $arg(t) \in \arg(w,y).$ If $(w,y)=(q,s),$ then Karpelevi\v c's theorem, in conjunction with Theorem \ref{thm:Type I}, shows that $t \in \partial \Theta_n.$ On the other hand, suppose that $(w,y) \ne (q,s)$. By Theorem 4 of \cite{Ki}, the boundary of $L_n$ in the sector  $\arg(w,y)$ is  given by roots of $t^y - \alpha t^{y-w}-(1-\alpha), \alpha \in (0,1)$. 
It now follows from Lemma 4 of \cite{Ki} that $t$ must be an interior point of $L_n,$ and hence $t$ is also an interior point of $\Theta_n$.  
}}
\end{remark}

In the general case we need to assume $d \geq 1$ and $\delta \geq 1$. Given $\theta \in \arg(q,s)$ we know that a root in $\partial \Theta_n$ with the argument $\theta$ is a root of $g_{\alpha,j}$ for some $j=0,\ldots,\delta-1$. We determined that $j$ is equal to $j_0$ for $\alpha=1$, and we need to show that this choice of $j$ does not change as $\alpha$ runs through the interval $(0,1)$. To prove this we need a few 
supporting results on the behaviour of roots on $\delta \Theta_n$. 

\begin{lemma}\label{lem:rho}
For each argument $\theta \in [0,2 \pi)$ and $n \geq 2$ there is unique $\rho$ so that $\rho e^{\ii \theta}$ is on the boundary of $\Theta_n$. 
\end{lemma}

\begin{proof} 
Suppose, to the contrary, that for some $\theta_0$ we have $\rho_1 e^{\ii \theta_0}$ and $\rho_2 e^{\ii \theta_0}$, $\rho_1 < \rho_2$, both on the boundary of $
\Theta_n$. Necessarily there is a Farey pair $(\sfrac{p}{q}, \sfrac{r}{s})$ of order $n$ such that $\theta_0 \in \arg(q,s).$ 
 From the fact that $\Theta_n$ is star--shaped with respect to the origin, it follows that for any $x$ with $\rho_1 < x < \rho_2$, $xe^{\ii \theta_0}$ must be on  the boundary of $\Theta_n$ (such a point is certainly in $\Theta_n$, and it cannot be in the interior of that set, otherwise $\rho_1 e^{\ii \theta_0}$ would also be an interior point).  But then for each such $x,$ there is a $j(x) \in \{0,\ldots,\delta-1\}$  such that $xe^{\ii \theta_0}$ is a root of $g_{\alpha,j(x)}$ for some $\alpha \in (0,1)$, and by Lemma \ref{lem:powers} there is a $l(x) \in \{0,\ldots,d_1-1\}$ such that $x^{\frac{1}{d_1}}e^{\ii \tau_{l(x)}}$ is a root of $\hat g_{\alpha, j(x)}$ for $\tau_{l(x)}=\frac{1}{d_1}(\theta_0+2 \pi l(x))$. 
In particular, there exist $j' \in \{0,\ldots,\delta-1\}$ and $l' \in \{0,\ldots,d_1-1\}$ such that $x^{\frac{1}{d_1}}e^{\ii \tau_{l'}}$ is a root of $\hat g_{\alpha, j'}$ for infinitely many choices of $x$.  
For all such $x$, we have by Lemma \ref{lem:phi} that $F_{\tau_{l'},j'}(x^{\frac{1}{d_1}})=0$:
\begin{align*}
&x^{\frac{s_1}{d_1}}\sin(q (\theta_0+2 \pi l'))-x^{q} \sin \left(\frac{s_1}{d_1}(\theta_0+2 \pi l')
-\frac{2 \pi j'}{\delta d_1} \right)-\\
&-\sin\left((q -\frac{s_1}{d_1})(\theta_0+2 \pi l') +\frac{2 \pi j'}{\delta d_1} \right)=0.
\end{align*}
 From the fact that $q$ and $s$ are relatively prime, it follows that at least one  of the coefficients of $x^{\frac{s_1}{d_1}}$ and $x^{q}$ in the equation above is nonzero. 
Evidently this equation cannot hold for infinitely many $x \in (\rho_1, \rho_2)$, a contradiction. 
	\end{proof}

\begin{lemma}\label{z1z2a}
Let $0<\rho< \min\{\rho',1\}$, $\rho' \ne 1$,  and $z \in \C$ with $|z|=\rho$.
There is at most one complex number $z'$ satisfying $|z'|=\rho'$, and 
$z - \alpha z' -(1-\alpha)=0$ for some $\alpha \in (0,1).$
\end{lemma}

\begin{proof}
Let $L$ be the line through $1$ and $z$, and let $C$ be the circle centred at the origin of radius $\rho'$. Observe that $L$ and $C$ intersect in precisely two points, say $z_1, z_2$.  Considering all possible positions of $1$ relative to $z$, $z_1$ and $z_2$ on the line $L$, and using the fact that $|z|=\rho < \rho'$ (Figure \ref{lfig:test} illustrates the two admissible placements of  $1, z, z_1$ and $z_2$),  
it is clear that $z$ cannot be both a convex combination of $z_1$ and $1$ and also a convex combination of  $z_2$ and $1$.  The conclusion now follows. 
\end{proof}

\begin{figure} [h!] \label{figz1} 
\centering
\begin{subfigure}{.5\textwidth}
  \centering
  \includegraphics[width=0.8\linewidth]{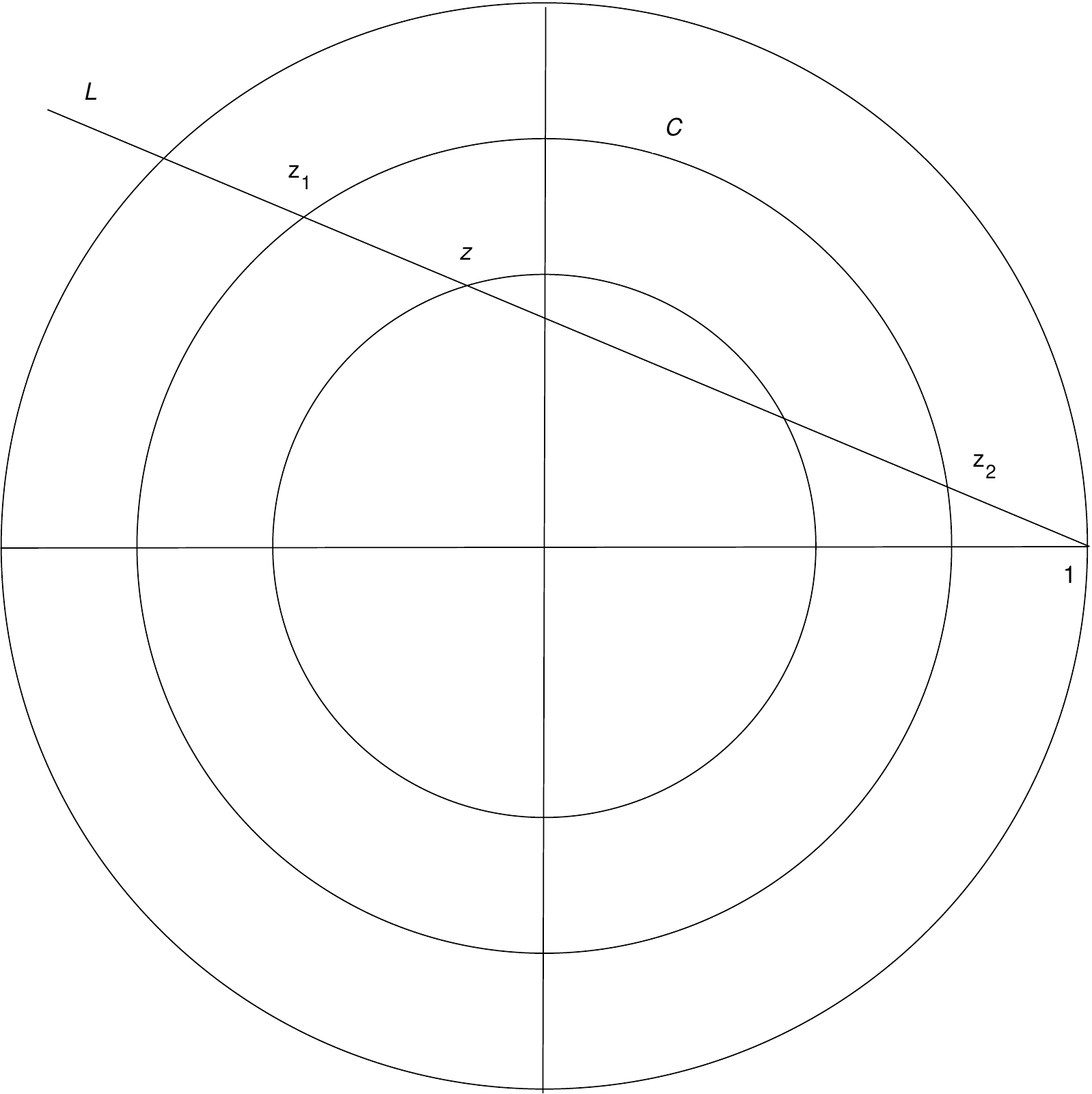} 
  \caption{$\rho' <1$}
  \label{fig:lsub1}
\end{subfigure}%
\begin{subfigure}{.5\textwidth}
  \centering
  \includegraphics[width=0.8\linewidth]{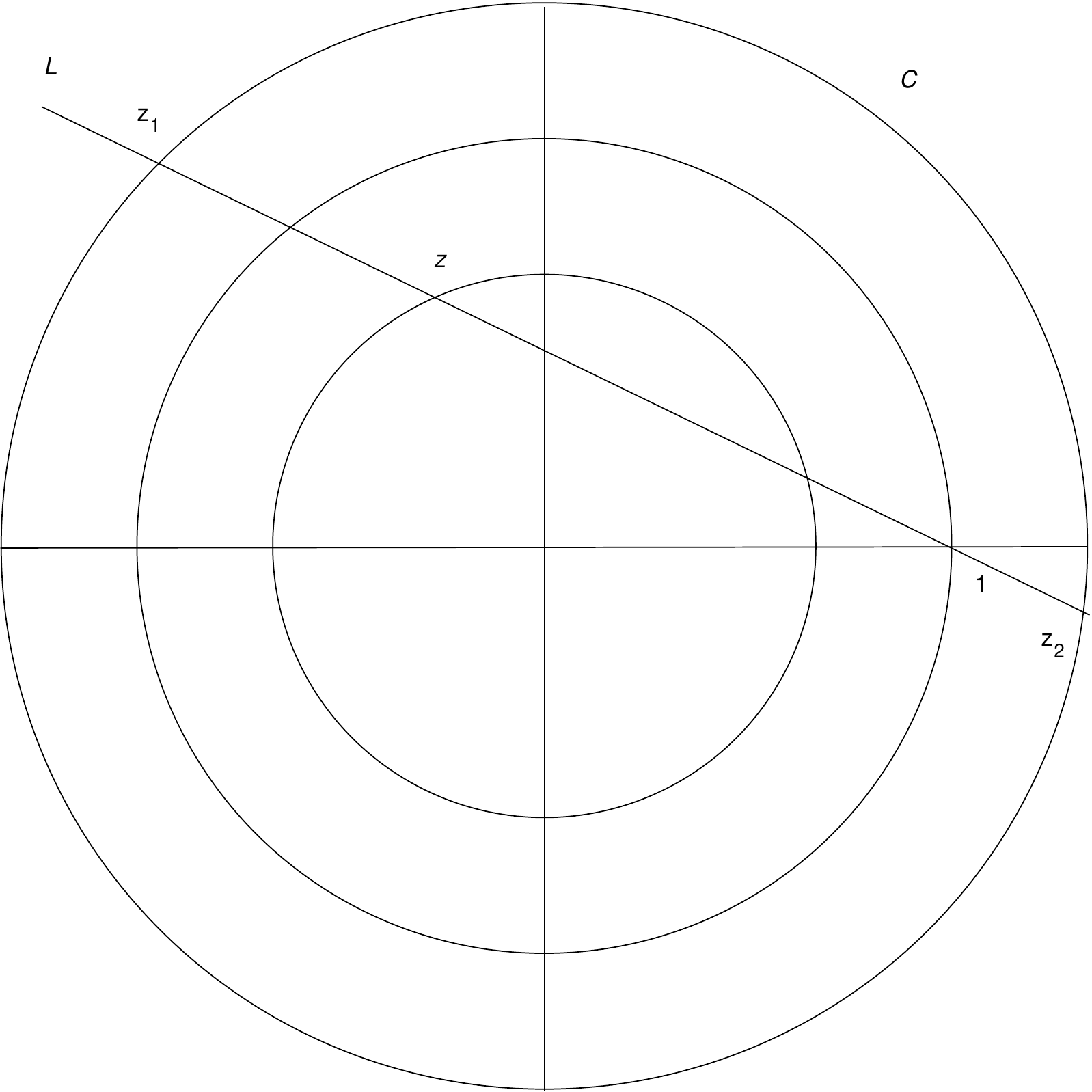} 
 \caption{$\rho' >1$}
  \label{fig:lsub2}
\end{subfigure}
\caption{\label{lfig:test} Relative placements of $1$, $z$, $z_1$ and $z_2$.}
\end{figure}



\begin{corollary}\label{cor:unique alpha}
Let $(\sfrac{p}{q},\sfrac{r}{s})$, $q \geq 2$, be a Farey pair, and let $t_0 \in \hat{\mc K}_n(q,s).$ Then $t_0$ is a root of the corresponding Ito rational function $\phi_{\alpha}$ for a single $\alpha \in (0,1)$. 
\end{corollary}

\begin{proof}
Let $\phi_{\alpha_1}(t_0)=\phi_{\alpha_2}(t_0)=0$ for some $\alpha_1,\alpha_2 \in (0,1)$. Hence, $g_{\alpha_1,j_1}(t_0)=g_{\alpha_2,j_2}(t_0)=0$ for some $j_1,j_2 \in \{0,\ldots,\delta-1\}$, and 
Lemma \ref{lem:powers} tells us that $t_0=\hat t_1^{d_1}=\hat t_2^{d_1}$, where $\hat t_1$ and $\hat t_2$ are roots of $\hat g_{\alpha_1,j_1}$ and  $\hat g_{\alpha_2,j_2}$, respectively. For $z\equiv \hat t_1^{q d_1}=\hat t_2^{q d_1}$, $z_1\equiv \hat t_1^{qd_1-s_1}e^{\frac{2 \pi \ii j_1}{\delta d_1}}$ and $z_2\equiv \hat t_2^{qd_1-s_1}e^{\frac{2 \pi \ii j_2}{\delta d_1}}$, we have $|z|<1$, $|z|<|z_1|=|z_2|$ and 
$$z-(1-\alpha_k) - \alpha_k z_k=0 \text{ for }k=1,2.$$  
By Lemma \ref{z1z2a}, we now know that $z_1=z_2$ and $\alpha_1=\alpha_2$.
\end{proof}

At this point we are ready to prove the main theorem in this section, first stated as Theorem \ref{thm:gcd200}, and restated below.

\begin{theorem}\label{thm:gcd2}
Suppose that  $(\frac{p}{q},\frac{r}{s})$, $q \leq s$, is a Farey pair in $\F_n$. For $\theta \in \arg(q,s)$ the point on the boundary of $\Theta_n$ with argument $\theta$ is given by $\hat \rho^{d_1} e^{\ii \theta}$ where $\hat \rho$ is the unique solution to \eqref{eq:qs1II} for $\hat \theta = \frac{1}{d_1}\left(\theta+2 \pi l_0\right)$ and $j=j_0$. Furthermore,  $\hat \rho^{d_1} e^{\ii \theta}$ is a root of $\phi_{\alpha}$, where $\alpha$ is given by \eqref{eq:qs2II} for $\rho=\hat \rho$ and $j=j_0$. 
\end{theorem} 

\begin{proof} 
Let $\theta \in \arg(q,s)$. By \eqref{eq:g} and Lemma \ref{lem:powers} we know that $\rho^{d_1}e^{\ii \theta}$ is a point on the boundary of $\Theta_n$, where $\rho$ is a solution to \eqref{eq:qs1II} for  $\tau =\kappa_l(\theta)\equiv \frac{1}{d_1}(\theta+2 \pi l)$ and some choices of $j \in \{0,\ldots,\delta-1\}$ and $l \in \{0, \ldots,d_1-1\}$. 
 

For $\theta=\frac{2 \pi r}{s}$ we have already established in Lemma \ref{lem:1} that $j=j_0$ and $l=l_0$. 
We define $\hat \rho(\theta)$ to be the unique solution to \eqref{eq:qs1II} for $\tau=\kappa_{l_0}(\theta)$ and $j=j_0$. Similarly, let  $\alpha(\theta)$ be defined by \eqref{eq:qs2II} for $\tau=\kappa_{l_0}(\theta)$ and $j=j_0$. Observe that  $\alpha(\theta)$   is continuous as a function of $\theta$ on $ [\frac{2 \pi p}{q}, \frac{2 \pi r}{s}].$ With this notation in place, define $t(\theta) =  \hat \rho(\theta)^{d_1} e^{\ii \theta}$ for $\theta \in [\frac{2 \pi p}{q}, \frac{2 \pi r}{s}]$.
  We deduce that there is an $\epsilon >0$ such that for all $\theta \in (\frac{2 \pi r}{s}-\epsilon, \frac{2 \pi r}{s}], t(\theta)$ is the point on $ \partial \Theta_n$ with argument $\theta.$ Let $$\theta_0 = \inf \left \lbrace \tau \Big| \frac{2 \pi p}{q}< \tau < \frac{2 \pi r}{s} \mbox{\rm{ and  }} t(\theta) \in \partial \Theta_n  \ \ \ \forall \theta \in \left(\tau, \frac{2 \pi r}{s}\right]   \right \rbrace.$$ We claim that $\theta_0 = \frac{2 \pi p}{q},$ which is sufficient to establish the desired conclusion. 
  
  To see the claim, suppose to the contrary that $\theta_0 > \frac{2 \pi p}{q}.$ Then there is a sequence $\tau_h \in \arg(q,s)$ such that $\tau_h \rightarrow \theta_0,$ but $t(\tau_h ) \notin \partial \Theta_n.$ Consequently there is a sequence of roots $y(\tau_h) \ne t(\tau_h)$ of $\phi_{\alpha_h}$ (with parameter $\alpha_h$) of the form $y(\tau_h)=\rho_h e^{\ii \tau_h}$ such that for each $h, y_h$ is the point on the boundary of $\Theta_n$ with argument $\tau_h.$
  Since $\partial \Theta_n$ is closed, by considering  subsequence  if necessary, we may assume that   $y(\tau_h)$ and $\alpha_h$ are convergent and that  $\lim_{h \rightarrow \infty} y(\tau_h) \equiv \tilde{\rho} e^{\ii \theta_0}$ is a root of $\phi_{\tilde \alpha}(t)$, where $\tilde \alpha\equiv\lim_{h \rightarrow \infty} \alpha_h$. It follows that $\tilde{\rho} e^{\ii \theta_0}$ is on the boundary of $\Theta_n.$

  By Lemma \ref{lem:rho} we know that $\tilde{\rho}=|t(\theta_0)|,$ hence $\tilde{\rho} e^{\ii \theta_0}=t(\theta_0)$. If $\alpha(\theta_0)=\tilde \alpha$,  then $\phi_{\alpha(\theta_0)}$ has a multiple root on the boundary of $\Theta_n,$ a contradiction to Theorem \ref{thm:double}. If on the other hand, $\alpha(\theta_0)\neq\tilde \alpha$, then $t(\theta_0)$ is a root of both $\phi_{\alpha(\theta_0)}$ and $\phi_{\tilde \alpha}$, contradicting Corollary \ref{cor:unique alpha}.
%
%
	We conclude that  $\theta_0 = \frac{2 \pi p}{q},$ as claimed. 
\end{proof}


\begin{remark}\label{rmk:gcd2} 
As we saw in Example \ref{ex:q=2},  it is possible for an lto rational function to have more than one root in the sector   $\arg(q,s)$ when $gcd(s,d) \ge 2.$  Theorem \ref{thm:gcd2} resolves the ambiguity in Theorem \ref{thm:Karpelevic} as to which root in $\arg(q,s)$ is on $\partial \Theta_n.$ Theorem \ref{thm:gcd2} also naturally   reparameterises the  boundary arc in terms of the argument $\theta $.

\end{remark}

\newpage

\section{Subdominance} \label{sec:sub} 

In this section we prove that for $n \ge 2,$ each point in $\Theta_n$ is a subdominant eigenvalue of some $n \times n$ stochastic matrix. Several technical results are required in order to establish that result. 
Throughout this section, we suppose that $(\frac{p}{q}, \frac{r}{s} )$ is a Farey pair in $\F_n$, and as before assume that $0 < \frac{p}{q}, \frac{r}{s} < 1.$ 
Set $\left \lfloor \frac{n}{q}\right \rfloor = d,$ and suppose that $gcd(d,s) = \delta$. Define $s_1, d_1$ via the equations $s=s_1 \delta, d=d_1 \delta.$   

\begin{lemma}
	\label{z1z2} 
	Suppose that $0<\alpha, \rho, \rho' <1.$ There is at most one pair of complex numbers $z, z'$ such that: i) $|z|=\rho$ and  $|z'|=\rho'$; ii) $\arg(z),\arg(z') \in [0, \pi]$; and iii) $z - \alpha z' -(1-\alpha)=0.$
\end{lemma}
\begin{proof}
	Suppose that there is a pair of complex numbers $z, z'$ satisfying i)--iii). Write $z=\rho e^{\ii\theta}, z'=\rho' e^{\ii\theta'}, $ where $\theta, \theta' \in  [0, \pi].$ From iii) we find that $\rho \cos(\theta) = \alpha \rho' \cos(\theta') + 1-\alpha$ and $\rho \sin(\theta) = \alpha \rho' \sin(\theta').$ Squaring and summing these equations yields $$\rho^2 = \alpha^2 \rho'^2 + (1-\alpha)^2 + 2\alpha(1-\alpha)\rho'\cos(\theta').$$ Consequently we see that 
	\begin{equation}\label{cosine}
	\cos(\theta')= \frac{\rho^2-\alpha^2\rho'^2-(1-\alpha)^2}{2\alpha(1-\alpha)\rho'}.
	\end{equation}
	Since the cosine function is strictly decreasing on $[0,\pi],$ it follows that $\theta'$ is uniquely determined by (\ref{cosine}). Hence $z'$ is uniquely determined, and we then find from iii) that $z$ is also uniquely determined. 
\end{proof}

\begin{theorem}\label{lem:modulus}
Let $\phi_{\alpha}$, $\alpha \in (0,1)$, be the Ito rational function corresponding to a Farey pair $(\sfrac{p}{q},\sfrac{r}{s})$, $s>q \geq 2$.
If $t_1$ and $t_2$ are roots of $\phi_{\alpha}$ such that $|t_1|=|t_2|,$ then either $t_2=t_1$ or $t_2=\overline{t_1}.$ 
\end{theorem}

\begin{proof} 
Let $t_1$ and $t_2$ be roots of $\phi_{\alpha}$ satisfying $|t_1|=|t_2|$, and as always let $\delta=\gcd(d,s)$. From \eqref{eq:factorphi} we find that there are indices $j_1, j_2$ between $0$ and $\delta-1$ such that $g_{\alpha,j_1}(t_1)=g_{\alpha,j_2}(t_2)=0$. By Lemma \ref{lem:powers} know that $t_1=\hat t_1^{d_1}$ and $t_2=\hat t_2^{d_1}$, where $\hat g_{\alpha,j_1}(\hat t_1)=\hat g_{\alpha,j_2}(\hat t_2)=0$. Hence, for $z_1=\hat t_1^{qd_1}$, $z_1'=\hat t_1^{qd_1-s_1}e^{\frac{2 \pi \ii j_1}{\delta d_1}}$,  $z_2=\hat t_2^{qd_1}$, and $z_2'=\hat t_2^{qd_1-s_1}e^{\frac{2 \pi \ii j_2}{\delta d_1}}$ we have:
\begin{align*}
z_1-\alpha z_1'-\beta&=0,  \\
z_2-\alpha z_2'-\beta&=0. \\
\end{align*}
Without loss of generality, we assume that $\arg(z_1), \arg(z_2) \in [0, \pi]$ (otherwise we consider one or two complex conjugates), and since $z_1$ is a convex combination of $1$ and $z_1'$ then necessarily  $\arg(z_1'),\arg(z_2') \in [0, \pi]$ as well. Since $|z_1|=|z_2|$ and $|z_1'|=|z_2'|$, we can apply Lemma \ref{z1z2}, and conclude that $z_1=z_2$ and $z_1'=z_2'$. 

	Since $\hat t_1^{q d_1} =\hat t_2^{q d_1},$ we have $\hat t_2^{d_1} = \hat t_1^{d_1} e^{\frac{2 \pi \ii b}{q}}$ for some $0 \le b \le q-1.$ Also, since $\gcd(qd_1, qd_1-s_1)=1,$ there are integers $x, y$ such that $x qd_1 + y(q d_1-s_1)=1.$ Hence $z_1^x z_1'^y = \hat t_1^{q d_1 x}\hat t_1^{(q d_1 -s_1)y}e^{\frac{2 \pi \ii j_1 y}{\delta d_1}} = \hat t_1e^{\frac{2 \pi \ii j_1 y}{\delta d_1}}.$ Similarly, we find that $z_2^x z_2'^y = \hat t_2 e^{\frac{2 \pi \ii j_2 y}{\delta d_1}},$ and hence  $\hat t_2^{d_1} = \hat t_1^{d_1} e^{\frac{2 \pi \ii c}{\delta}}$ for some $0 \le c \le \delta-1.$ 
	Recalling that $\delta|s$ and $\gcd(s,q)=1,$ we find that $\gcd(\delta,q)=1.$ Since $ e^{\frac{2 \pi \ii  b}{q}} = e^{\frac{2 \pi \ii c}{\delta}},$ we see that $\frac{b}{q} - \frac{c}{\delta} \in \mathbb{Z},$ i.e. $b\delta-cq=q\delta h$ for some $h \in \mathbb{Z}.$ We find readily that $q|b$ and $\delta|c,$ and this in turn yields the fact that $\hat t_1^{d_1} = \hat t_2^{d_1},$ i.e. $t_1=t_2$. 	
\end{proof}

\begin{lemma} \label{deriv} Suppose that $(\frac{p}{q}, \frac{r}{s})$ is a Farey pair with $0< \frac{p}{q}, \frac{r}{s} < 1$  $q<s$, $\hat g_{\alpha,j}$, $j=0,\ldots,\delta-1$, as defined in \eqref{eq:hg} for this pair, and $m \in \{1,\ldots, q d_1-1\}$. 
Let $\rho(\alpha) e^{\ii \theta(\alpha)}$ denote the locus of roots of $\hat g_{\alpha,j}$ that coincides with $e^{\frac{2 \pi \ii m}{q d_1}}$ when $\alpha =0$, and let $\kappa(m,j):=\frac{2 ms_1}{qd_1} - \frac{2 j}{\delta d_1}.$  We have 
	$$\frac{d \rho}{d \alpha}\Big |_{\alpha=0} = 
	\begin{cases}
	\frac{-1}{q d_1}\left(1- \cos \left(\kappa(m,j)\pi\right)\right),  &\text{ for } 
	\kappa(m,j) \notin \mathbb{Z},\\
	 \frac{-2}{q d_1}, &\text{ for } 
	\kappa(m,j) \in \mathbb{Z}, 
	\end{cases}
	$$
	where $\frac{d \rho}{d \alpha}\Big |_{\alpha=0}$ is interpreted as the derivative from the right. 
\end{lemma}
\begin{proof} 
Recall that $\rho, \theta$ and $\alpha$ satisfy \eqref{eq:qs1II} and \eqref{eq:qs2II}. 
	Implicitly differentiating \eqref{eq:qs1II} and \eqref{eq:qs2II}, we find that 
	\begin{align*}
	 &\frac{d \rho}{d \theta}\Big |_{\theta =\frac{2 \pi m}{qd_1} }\sin (\pi  \kappa(m,j) )=1-\cos (\pi  \kappa(m,j) ) \\
	 &\frac{d \theta}{d \alpha}\Big |_{\alpha=0}d_1 q=-\sin (\pi  \kappa(m,j) ),
	\end{align*}
and hence 
      $$\frac{d \rho}{d \alpha}\Big |_{\alpha=0}\sin (\pi  \kappa(m,j) )d_1 q=-\sin (\pi  \kappa(m,j) )(1-\cos (\pi  \kappa(m,j) ) ).$$
If $\kappa(m,j) \notin \mathbb{Z},$ so that 
	$\sin(\pi \kappa(m,j) ) \ne 0$,  the desired expression for $\frac{d \rho}{d \alpha}\Big |_{\alpha=0}$ follows.

Next we consider the case that  $a:= \kappa(m,j)  \in \mathbb{Z}.$  From $2ms_1\delta-2jq=aq\delta d_1$ and   $\gcd(\delta,q)=1,$ we find that $\delta|2j;$ since $0 \le 2j \le 2\delta-2,$ it must be the case that either $j=0$ or $\delta$ is even and $j=\frac{\delta}{2}.$

If $j=0,$ then then $qd_1|2ms_1,$ and since both $q$ and $d_1$ are relatively prime to $s_1,$ we find that $qd_1|2m.$ From the hypothesis that $1 \le m \le qd_1-1,$ it follows that $qd_1=2m,$ so that at $\alpha=0,$ our root $ e^{\frac{2\pi \ii m}{qd_1}}=-1.$ Since $qd_1 $ is even, $s_1$ must be odd and so for $\alpha$ in a neighbourhood of $0$,  our root is of the form $-\rho,$ where $\rho$ is real, positive and solves $\rho^{qd_1}+\alpha \rho^{qd_1-s_1}-(1-\alpha)=0$. Differentiating with respect to $\alpha$ now yields $\frac{d \rho}{d \alpha}\Big |_{\alpha=0} = \frac{-2}{q d_1}.$ 
	
	Finally, suppose that  $\delta$ is even and $j=\frac{\delta}{2}.$ Then for some integer $a$ we have $2ms_1-q=aqd_1.$ Since $\delta$ is even, $q$ and $d_1$ are odd, and we deduce that $a$ is also odd. In this case, our polynomial $\hat g_{\alpha,j}(t)$ simplifies to $t^{q d_1} - (1-\alpha) - \alpha t^{q d_1 -s_1}e^{\frac{ \pi \ii}{d_1}}.$ It follows that $\hat g_{\alpha,j}$ has a root of the form $\rho e^{\frac{2 \pi \ii m}{q d_1}}$ (where $\rho$ is real and positive) if and only if $\rho^{q d_1}-\alpha \rho^{qd_1 -s_1}  e^{\frac{-2 \pi \ii ms_1}{q d_1} + \frac{\pi \ii}{d_1}} -(1-\alpha)=0.$ Since $2ms_1-q=aqd_1,$ this last reduces to $\rho^{q d_1}+\alpha \rho^{qd_1-s_1}-(1-\alpha)=0$ (which has a unique positive solution in $(0,1]$ for all $\alpha \in [0,1)$). Differentiating as before, it follows that $\frac{d \rho}{d \alpha}\Big |_{\alpha=0} = \frac{-2}{q d_1}.$ 
%
%
%
%
%
\end{proof}

In the following, we maintain the notation and terminology of Lemma \ref{deriv}. 

\begin{corollary}\label{II_der} 
	We have 
	\begin{equation}\label{II_bd}
	\frac{d \rho}{d \alpha}\Big |_{\alpha=0} \le \frac{-1}{q d_1}\left(1-\cos \left(\frac{2 \pi}{q\delta d_1}\right)\right). 
	\end{equation}
	Further, if equality holds in (\ref{II_bd}), then 
	there is an $1 \ge \alpha_0>0$ such that for all $\alpha \in [0, \alpha_0),$ we have 
	either:
	\begin{enumerate}
	 \item 
	 $(\rho(\alpha)e^{\ii\theta(\alpha)})^{d_1} = (\hat{\rho}(\alpha)e^{\ii\hat{\theta}(\alpha)})^{d_1},$ where  $\hat{\rho}(\alpha)e^{\ii\hat{\theta}(\alpha)}, \alpha \in [0, \alpha_0)$ is  the locus of roots of $\hat g_{\alpha,j_0}$ that equals $e^{\frac{2 \pi \ii (p+q l_0)}{q d_1}}$ when $\alpha=0,$ or
	\item  $(\rho(\alpha)e^{\ii\theta(\alpha)})^{d_1} = \overline{(\hat{\rho}(\alpha)e^{\ii\hat{\theta}(\alpha)})^{d_1}},$
	where 
	 $\hat{\rho}(\alpha)e^{\ii\hat{\theta}(\alpha)}, \alpha \in [0, \alpha_0)$ is as in 1.  
	 \end{enumerate}
	 In particular,   if equality holds in (\ref{II_bd}), then 
	 $(\rho(\alpha)e^{\ii\theta(\alpha)})^{d_1} \in \hat{\mc K}_n(q,s) $ when $\alpha \in (0, \alpha_0).$
\end{corollary}

\begin{proof}
	Evidently (\ref{II_bd}) holds if $\kappa(m,j) \in \mathbb{Z},$ so henceforth we focus on the case that $\kappa(m,j) \notin \mathbb{Z}.$ Suppose that $2ms_1\delta-2jq \equiv x \mod 2q\delta d_1,$ and notice that necessarily $x$ must be even and nonzero. It now follows that
	$$\cos \left(\kappa(m,j)\pi\right) = \cos \left(\frac{\pi x}{q\delta d_1} \right) \le  \cos \left(\frac{2 \pi}{q\delta d_1}\right),$$
 and inequality (\ref{II_bd}) follows. 
Inspecting the argument above, we find that equality holds in (\ref{II_bd}) if and only if $x=2$ or $x=2q\delta d_1-2$, i.e. if and only if  $ms_1\delta-jq \equiv \pm 1 \mod q\delta d_1.$ We now proceed to characterise that situation.

Since $(\frac{p}{q}, \frac{r}{s})$ is a Farey pair, we have $qr-ps=1$, and we find that $ms_1\delta-jq \equiv \pm (rq-ps) \mod q\delta d_1$.
Write $m=m_1+m_2q,$ where $0 \le m_1 \le q-1$ and $0 \le m_2 \le d_1-1.$ We first consider the case 
\begin{equation}\label{eq:case1}
m_1s_1\delta+m_2 q s_1\delta-jq \equiv - (rq-ps) \mod q\delta d_1.
\end{equation} 
Considering \eqref{eq:case1} modulo $q$ and taking into the account that $\gcd(q,s)=1$, we find that $m_1 \equiv p \mod q$, hence $m_1=p$. Similarly, considering \eqref{eq:case1} modulo $\delta$ together with $\gcd(q,\delta)=1$ implies $j \equiv r \mod \delta$, hence $j=j_0$ as defined in Definition \ref{def:deltaetc}. Now \eqref{eq:case1} reduces to $m_2s_1 \equiv -r_2 \mod d_1$. Since $s_1$ and $d_1$ are relatively prime and $0 \le m_2 \le d_1-1$, it follows that there is a unique such $m_2,$ namely $m_2=l_0.$ We deduce that for some $\alpha_1>0,$ for $\alpha \in [0, \alpha_1), \rho(\alpha)e^{\ii \theta(\alpha)}$ is the continuous locus of roots of $\hat{g}_{\alpha,j_0}$ that coincides with $e^{\frac{2 \pi \ii (p+ql_0)}{qd_1}}$ at $\alpha=0.$


	Consider the interval $[\frac{2 \pi p}{qd_1}+ \frac{2 \pi l_0}{d_1},  \frac{2 \pi r}{sd_1}+ \frac{2 \pi l_0}{d_1}],$ and note that if $\hat{\theta}$ lies in this interval, we find from Lemma \ref{lem:phi} that there is a unique $\hat{\rho}>0$ and a corresponding value $\alpha \in [0,1]$ such that $\hat{\rho} e^{\ii \hat{\theta}}$ is a root of $\hat{g}_{\alpha,j_0}$, and $(\hat{\rho} e^{\ii \hat{\theta}})^{d_1} \in  \hat{\mc K}_n(q,s).$ Further, from (\ref{eq:qs1II}) and  (\ref{eq:qs2II}), it follows that for some $\alpha_0 \in (0,\alpha_1),$ $\hat{\rho}$ and $\hat{\theta}$ can be written as continuous functions of $\alpha \in [0,\alpha_0 ].$ It now follows that for $\alpha \in [0,\alpha_0 ],$ 
 we have 
	$\rho(\alpha)e^{\ii\theta(\alpha)} ={ \hat \rho(\alpha)}e^{\ii {\hat \theta(\alpha)}},$ in accordance with 1). 


	In the case that  $ms_1\delta-jp \equiv  1 \mod q\delta d_1,$ set $p'=q-p, r'=s-r,$ and observe that  $ms_1\delta-jp \equiv p's-r'q \mod q\delta d_1.$ Arguing as above, it follows that $j=j_0',$ where  $r' = \delta \lfloor \frac{r'}{\delta}\rfloor + j_0'$ and $m= p' + l_0'q,$ where  $l_0'$ satisfies $ \lfloor \frac{r'}{\delta}\rfloor = d_1 \hat{r'} - l_0's_1,$ with $0 \le l_0' \le d_1-1, 0 \le \hat{r'} \le s_1-1.$ 
	As above we find that for some $\alpha_1>0, $  when $\alpha \in [0, \alpha_1), \rho(\alpha)e^{\ii \theta(\alpha)}$ is the continuous locus of roots of $\hat{g}_{\alpha,j_0'}$ that coincides with $e^{\frac{2 \pi \ii (q-p+ql_0')}{qd_1}}$ at $\alpha=0.$ Further, we find that for some $\alpha_0 \in (0, \alpha_1),$ when for $\alpha \in [0,\alpha_0 ],$ we have 
	$\rho(\alpha)e^{\ii\theta(\alpha)} =\hat{ \rho}(\alpha)e^{\ii \hat{\theta}(\alpha)},$ 
	where $\hat{\rho}, \hat{\theta}$ and $\alpha$ satisfy  (\ref{eq:qs1II}),   (\ref{eq:qs2II}) and (\ref{eq:theta}), when $j=j_0', l_0=l_0'$. The conclusion for ii) now follows. 
\end{proof}

Recall that for an Ito rational function $\phi_\alpha, \alpha \in [0,1],$ we always have $1$ as a root of maximum modulus. We say that a root $z$ of  $\phi_\alpha$ is subdominant if its modulus is next largest after $1$. 

\begin{theorem}\label{typeIIthm} 
	Let $(\frac{p}{q}, \frac{r}{s})$ be a Farey pair in $\F_n$ with the corresponding Ito rational function $\phi_{\alpha}$, and assume that $0< \frac{p}{q}, \frac{r}{s} <1,$ and $q<s.$   
For $\theta \in \arg(q,s)$ let $ \rho e^{\ii \theta}$ be the root of $\phi_{\alpha(\theta)}$ that is on the boundary of $\Theta_n$.
		Then $ \rho e^{\ii \theta}$
		is a subdominant root of $\phi_{\alpha(\theta)}.$ 
\end{theorem}
\begin{proof} 
	By Theorem \ref{thm:gcd2} we know that $\rho=\hat \rho^{d_1}$, where  $\hat \rho$ is the unique solution to \eqref{eq:qs1II} for $\hat \theta = \frac{\theta+2\pi l_0}{d_1}$ and $j=j_0$. Let $t(\theta)=\hat \rho^{d_1}e^{\ii \theta}  $ for  $\theta \in [\frac{2 \pi p}{q}, \frac{2 \pi r}{s}]$, and let $\alpha(\theta)$ be given by \eqref{eq:qs2II} for $\hat \rho$, $\hat \theta= \frac{\theta+2\pi l_0}{d_1}$ and the same choice  $j=j_0$.

	Observe first that $\alpha(\hat \theta)$ is continuous in $\theta$, that when $\theta = \frac{2 \pi p}{q}$ we have $\hat \rho=1, \alpha(\theta)=0,$ and when $\theta = \frac{2 \pi r}{s}$ we have $\hat \rho=1, \alpha(\theta)=1.$  Define  $\tilde{\phi}_{\alpha(\theta)}(t) \equiv  \frac{\phi_{\alpha(\theta)}(t)}{(t-t(\theta))(t-\overline{t(\theta)})(t-1)}.$ 
	Applying Corollary \ref{II_der}, we conclude that   
	there is a positive $\epsilon$ such that 
	such that for all $\theta \in (\frac{2 \pi p}{q}, \frac{2 \pi p}{q}+\epsilon),$ and any root $z$ of $\tilde{\phi}_{\alpha(\theta)}$, we have $|z|<|t(\theta)|.$ 
	
	Let $$\theta_0 \equiv  \sup \{ \theta \in [\sfrac{2 \pi p}{q}, \sfrac{2 \pi r}{s}] | t(\nu) {\mbox{\rm{ a subdominant root of }}} {\phi}_{\alpha(\nu)}  {\mbox{ for }} \nu \in [\sfrac{2 \pi  p}{q}, \theta)\}.$$ If $\theta_0 = \frac{2 \pi r}{s},$ we are done, so suppose that $\theta_0 < \frac{2 \pi r}{s}.$ 
	Since $t(\theta)$ is a subdominant root of ${\phi}_{\alpha(\theta)}$ for $\theta \in [\frac{2 \pi  p}{q}, \theta_0),$ it follows that necessarily $t(\theta_0)$ is a subdominant root of ${\phi}_{\alpha(\theta_0)}.$ Also, there is a sequence $\theta_h \in (\theta_0, \frac{2 \pi r}{s}]$ such that $\theta_h \rightarrow \theta_0$ as $h \rightarrow \infty,$ and in addition, for each $h \in \mathbb{N},$ there is a root $z(\theta_h)$ of $\tilde{\phi}_{\alpha(\theta_h)}$  such that $|z(\theta_h)| > |t(\theta_h)|$. Passing to convergent subsequences if necessary, we find that there is a root $z_0$ of  $\tilde{\phi}_{\alpha(\theta_0)}$  such that $|z_0| \ge |t(\theta_0)|$; as $t(\theta_0)$ is a subdominant root, it must be the case that $|z_0| = |t(\theta_0)|.$

	Then by Theorem \ref{lem:modulus}, we have either  $z_0=t(\theta_0)$ or $z_0=\overline{t(\theta_0)}.$ In either case, $t(\theta_0)$ is a multiple root of $\phi_{\alpha(\theta_0)}.$ This last is  a contradiction to Theorem \ref{thm:double}, since  $t(\theta_0)$ is on the boundary of $\Theta_n.$   It now follows that for all $\theta \in [\frac{2 \pi p}{q},\frac{2 \pi r}{s}], t(\theta)$ is  a subdominant root of $\phi_{\alpha(\theta)}.$ 
\end{proof}

\begin{corollary}\label{cor:subdom} Suppose that $n \in \N $ with $n \ge 2,$ and that $z \in \Theta_n.$ Then $z$ is a subdominant eigenvalue of a stochastic matrix of order $n$. 
	Further, if $z \ne 1,$ then $z$ is a subdominant eigenvalue of an irreducible  stochastic matrix of order $n$, and if $z \ne  e^{\frac{2 \pi \ii p}{q}}$  for any   $\sfrac{p}{q} \in \F_n,$  then $z$ is a subdominant eigenvalue of a  primitive   stochastic matrix of order $n$. 
\end{corollary}
\begin{proof}	
	First, suppose that we have a stochastic matrix $T$ of order $m<n.$ Let ${\bf{1}}_{k}$ denote the all ones vector of order $k,$ let $I_m$ be the identity matrix of order $m,$ and let $X, Y$ be the $n\times m$ and $m \times n$ matrices given by $$X=\left[\begin{array}{c|c} {\bf{1}}_{n-m+1}& 0\\ \hline 0 & I_{m-1}\end{array}\right],Y=\left[\begin{array}{c|c} \frac{1}{n+1-m}{\bf{1}}_{n-m+1}^{\top}& 0^{\top}\\ \hline 0 & I_{m-1}\end{array}\right].$$   It is straightforward to determine that $YX=I_m$ and that $\tilde{T} = XTY$ is a stochastic matrix of order $n$, with eigenvalues equal to the eigenvalues of $T,$ along with $n-m$ additional zero eigenvalues.  Further, since $\tilde{T}^k = XT^k Y, k \in \N$, we find that $\tilde{T}$ is irreducible if $T$ is irreducible, and that $\tilde{T}$ is primitive if $T$ is primitive. It now follows that if $z$ is a subdominant eigenvalue of the stochastic matrix $T$ of order $m<n,$ then $z$ is also   a subdominant eigenvalue of the stochastic matrix $\tilde{T}$ of order $n$, and further that the irreducibility (respectively, primitivity) of $\tilde{T}$ is inherited from $T$.

		Fix a point $z \in \partial \Theta_n.$ If $z = 1,$ then it is certainly a subdominant eigenvalue of the $n \times n$ identity matrix. Also, if $z =   e^{\frac{2 \pi \ii p}{q}}$  for some  $\sfrac{p}{q} \in \F_n,$ then $z$ is a subdominant eigenvalue of an irreducible cyclic permutation matrix of order $q$, and so from the argument above we find that $z$ is also a subdominant eigenvalue of  an irreducible stochastic matrix of order $n$. 	
		Henceforth we assume that $z \ne 1$ and $ z \ne e^{\frac{2 \pi \ii p}{q}}$  for any  $\sfrac{p}{q} \in \F_n.$ 
		
		If $\arg(z)  \in (0, \sfrac{2 \pi }{n}) $, then by Theorem \ref{thm:Karpelevic}, $z$ has the form $1-\alpha + \alpha \cos(\sfrac{2 \pi}{n})$ for some $\alpha \in (0,1).$ It is now readily verified that $z$ is a subdominant eigenvalue of $(1-\alpha)I + \alpha P,$ where $P$ is an $n \times n$ cyclic permutation matrix. An analogous argument applies if $\arg(z) \in (\sfrac{2 \pi (n-1)}{n}, 2 \pi).$ 
		
		If there is a Farey pair $(\sfrac{p}{q}, \sfrac{r}{s}) $ in $\F_n$ such that $0 < \sfrac{p}{q}, \sfrac{r}{s} < 1,$ and $\arg(z) \in \arg(q,s)$, then it follows from  Theorem \ref{typeIIthm} that $z$ is a subdominant root of  the corresponding reduced Ito polynomial $f_\alpha,$ for some $\alpha \in (0,1).$ From the proof of  Theorem 3.2 of \cite{JP}, there is a primitive stochastic matrix $T$ of order $s$ whose characteristic polynomial is $f_\alpha$; it now follows from the argument above that there is a primitive $n \times n$ stochastic matrix $\tilde{T}$ having $z$ as a subdominant eigenvalue.

	Finally, suppose that $z$ is an interior point of $\Theta_n.$ Then for some $c \in [0,1),$ and some point $z_0$ on the boundary of $\Theta_n,$ we have $z=cz_0.$ Then there is a stochastic matrix $T$ of order $n$ having $z_0$ as a subdominant eigenvalue. For concreteness, denote the eigenvalues of $T$ by  $1, \lambda_2, \ldots, \lambda_n$ (evidently $z_0=\lambda_j $ for some $j=2, \ldots, n$). By a theorem of Brauer \cite{Brauer}, the eigenvalues of the stochastic matrix $c T + (1-c){\bf{1}}_n{\bf{1}}_n^{\top}$ are $1, c \lambda_2, \ldots, c \lambda_n$.  We thus find that  $z=cz_0$ is a subdominant eigenvalue of $c T + (1-c){\bf{1}}_n{\bf{1}}_n^{\top}$, which is $n \times n,$ stochastic, and primitive. 
\end{proof}

\bibliographystyle{elsarticle-num} 
\bibliography{../KarpelevicRef}

\end{document}